\documentclass[11pt,reqno]{amsart}
\usepackage{amsmath,mathrsfs}

\usepackage{enumerate}
\usepackage[english]{babel}
\usepackage{hyperref}
\usepackage[all,cmtip]{xy}
\entrymodifiers={+!!<0pt,\fontdimen22\textfont2>}

\theoremstyle{theorem}
\newtheorem{thm}{Theorem}[section]
\newtheorem{lem}[thm]{Lemma}
\newtheorem{prop}[thm]{Proposition}
\newtheorem{cor}[thm]{Corollary}

\theoremstyle{definition}
\newtheorem{defn}[thm]{Definition}
\newtheorem{ex}[thm]{Example}
\newtheorem{notation}[thm]{Notation}
\theoremstyle{remark}
\newtheorem{rem}[thm]{Remark}
\newtheorem*{ack}{Acknowledgement}

\def\Z{{\mathbb Z}} %integer
\def\C{{\mathbb C}} %complex
\def\R{{\mathbb R}} %real

\def\K{{\mathcal K}} %compact

\mathchardef\ordinarycolon\mathcode`\: 
\def\vcentcolon{\mathrel{\mathop\ordinarycolon}} 
\providecommand*\coloneqq{\mathrel{\vcentcolon\mkern-1.2mu}=}

\def\cast{$C^{*}$}
\DeclareMathOperator\id{id} %identity
\DeclareMathOperator*\colim{colim}%colimit
\def\incl{\hookrightarrow}
\def\m{{\mathbf m}}

\def\S{\mathbf S}
\def\n{{\mathbf n}}
\def\Ho{\mathbf{Ho}}

\def\bu{\mathbf{bu}}
\def\kk{\mathbf{kk}}

\begin{document}

\title{Unsuspended Connective $E$-theory}
\author{Otgonbayar Uuye}
\date{\today}   

\address{
Department of Mathematical Sciences\\
University of Copenhagen\\
Universitetsparken 5\\
DK-2100 Copenhagen E\\
Denmark}
\email{otogo@math.ku.dk}
\urladdr{http://www.math.ku.dk/~otogo}
%\keywords{transversally elliptic, index theory}
%\subjclass[2010]{Primary (19K56); Secondary (46L87)}

\begin{abstract}
We prove connective versions of results by Shulman \cite{MR2606883} and D\u{a}d\u{a}rlat-Loring \cite{MR1305073}. As a corollary, we see that two separable $C^*$-algebras of the form $C_0(X) \otimes A$, where $X$ is a based, connected, finite CW-complex and $A$ is a unital properly infinite algebra, are $\bu$-equivalent if and only if they are asymptotic matrix homotopy equivalent.  
\end{abstract}

\maketitle
\setcounter{section}{-1}
\section{Introduction}

Let $\S$ denote the Connes-Higson {\em asymptotic homotopy category} of separable \cast-algebras (c.f.\ \cite{MR1065438, MR1711324}). Let $\Sigma$ denote the suspension functor $\Sigma B \coloneqq C_0(\R) \otimes B$ and let $\K$ denote the algebra of compact operators on a separable Hilbert space. 

$E$-theory is the bivariant $K$-theory defined by  
	\begin{equation}
	E(A, B) \coloneqq \S(\Sigma A, \Sigma B \otimes \K).
	\end{equation}

In this paper, we prove {\em connective} extensions of the following two closely related results. %\footnote{It is not difficult to deduce Theorem~\ref{thm shulman} from Theorem~\ref{thm DL}. This paper shows that we can also deduce Theorem~\ref{thm DL} from Theorem~\ref{thm shulman}.} results. 
\begin{thm}[{Shulman \cite{MR2606883}}]\label{thm shulman} Let $A$ be a separable \cast-algebra. Then $qA \otimes \K$ is $\S$-equivalent to $\Sigma^2 A \otimes \K$.
\end{thm}

\begin{thm}[{D\u{a}d\u{a}rlat-Loring \cite[Theorem 4.3]{MR1305073}}]\label{thm DL} Let $A$ and $B$ be separable \cast-algebras. If the abelian monoid $\S(A, A \otimes \K)$ is a group, then the suspension functor induces an isomorphism 
	\begin{equation}
	\S(A, B \otimes \K) \cong E(A, B \otimes \K).
	\end{equation}
\end{thm}

See Theorem~\ref{thm con Shul} and Theorem~\ref{thm con DL} for the precise statements.	 Considering {\em stable} algebras, we obtain Theorem~\ref{thm shulman} and Theorem~\ref{thm DL}, respectively. We note that this gives new\footnote{Our proof of Theorem~\ref{thm DL} is closely related to the remark at end of Section 4 in \cite[Theorem 4.3]{MR1305073}.} and more conceptual, if not simpler, proofs of the theorems.

We refer to \cite{andreasthomE} and references therein for details of connective $E$-theory and its applications.

\begin{ack} This research is supported by the Danish National Research Foundation (DNRF) through the Centre for Symmetry and Deformation at the University of Copenhagen. The author wishes to thank Tatiana Shulman and Hannes Thiel for stimulating conversations on the topics presented here. The author also wishes to thank Takeshi Katsura for insightful comments.
\end{ack}

\section{Asymptotic Matrix Homotopy Category}

We start by fixing some notation.

\begin{notation} 
\begin{enumerate}[(i)]
\item Let $A$ and $B$ be \cast-algebras. We write $A \star B$, $A \times B$ and $A \otimes B$ for the free product, direct product/sum and maximal tensor product of $A$ and $B$, respectively.

\item For $n \ge 1$, let $M_{n}$ denote the \cast-algebra of $n \times n$ complex matrices. For $n$, $m \ge 1$, we write $\oplus$ for the operation 
	\begin{align}
	\oplus \colon M_{n} \times M_{m}  &\to M_{n+m}, \quad (a, b) \mapsto \begin{pmatrix}a & \\ & b\end{pmatrix}
	\end{align}
and, $i_{m, n}$, for $m \ge n$, for the inclusion
	\begin{align}
	i_{m, n}\colon M_{n} &\incl M_{m}, \quad a \mapsto a \oplus 0.
	\end{align} 
%For $m \ge n \ge 1$, we write $i_{m, n}\colon M_{n} \incl M_{m}$ for the inclusion $a \mapsto a\oplus 0$.% and $i_{n,1}\colon \C \incl M_{n}$ for the composition $j_{n-1}\circ \dots \circ j_1$. 
We identify $\C$ with $M_1$ and $M_{n} \otimes M_{m}$ with $M_{nm}$ for $n$, $m \ge 1$, and $\K$ with the colimit of $M_{n}$ along $i_{m, n}$.% and write $i_{\infty}$ for the colimit inclusion $\C \incl \K$.
%We write $i_{n+1,n}\colon M_{n} \incl M_{n+1}$ for the inclusion $a \mapsto a\oplus 0$ and $i_{n,1}\colon \C \incl M_{n}$ for the composition $j_{n-1}\circ \dots \circ j_1$. 
%
%We identify $\K$ with the colimit of $M_{n}$ along $j_n$ and write $i_{\infty}$ for the colimit inclusion $\C \incl \K$.
\item For $k \ge 0$, let $\Sigma^{k}$ denote the \cast-algebra $C_{0}(\R^{k})$ of continuous functions on $\R^k$ vanishing at infinity. We identify $\Sigma^0$ with $\C$ and $\Sigma^k \otimes \Sigma^l$ with $\Sigma^{k+l}$ for $k$, $l \ge 0$.
\end{enumerate}
\end{notation}

\begin{defn} Let $A$ and $B$ be separable \cast-algebras. We define $\m(A, B)$ as the  colimit
	\begin{equation}
	\m(A, B) \coloneqq \colim_{n \to \infty} \S(A, B \otimes M_{n})
	\end{equation}
along $(\id_{B} \otimes i_{m, n})_{*}$.
\end{defn}

We summarize some properties of $\m$ that are well-known and/or easy to check. Statements (\ref{m1})-(\ref{m3}) say, essentially, that $\m$ is a homotopy invariant, matrix stable category enriched over the abelian monoids.

\begin{prop} Let $A$, $B$, $C$ and $D$ stand for separable \cast-algebras and let $m$, $n \ge 1$.
\begin{enumerate}[(i)]
\item\label{m1} Homotopic $\ast$-homomorphisms $A \to B$ define the same class in $\m(A, B)$.
\item The composition 
	\begin{align}
	\S(B, C \otimes M_{m}) &\times \S(A, B \otimes M_{n})  \to \S(A, C \otimes M_{mn})\\
	(g&, f) \mapsto (g \otimes \id_{M_{n}})\circ f
	\end{align}
gives $\m$ a category structure, with the identity morphism on $A$ represented by $\id_A \otimes i_{n, 1}\colon A \to A \otimes M_{n}$.
\item\label{m3} The addition
	\begin{align}
	\S(A, B \otimes M_{n}) &\times \S(A, B \otimes M_{m}) \to \S(A, B \otimes M_{n + m})\\
	(f&, g) \mapsto f \oplus g
	\end{align}
gives $\m(A, B)$ the structure of an abelian monoid, bilinear with respect to composition.
\item\label{m4} The tensor product 
	\begin{align} 
	\S(A, B \otimes M_{n}) &\times \S(C, D \otimes M_{m}) \to \S(A \otimes C, B \otimes D \otimes M_{nm})\\
	 (f&, g) \mapsto f \otimes g
	\end{align}
defines a natural bilinear functor 
	\begin{align} 
	\otimes\colon \m(A, B) \times \m(C, D) \to \m(A \otimes C, B \otimes D).
	\end{align}
\item\label{m5} For any $A$ and $C$, the functor $F(B) \coloneqq \m(A, B \otimes C)$ is split exact.
\end{enumerate}
\end{prop}
\begin{proof}
%We prove only the last two statements. The statement (\ref{m4}) follows from \cite[Theorem 4.6]{GHT} and (\ref{m5} )follows from \cite[Proposition 3.2]{MR1305073} and \cite[Theorem 2.6.15]{weibel}.
We prove only the last statement (\ref{m5}). This follows from \cite[Proposition 3.2]{MR1305073} and \cite[Theorem 2.6.15]{MR1269324}.
\end{proof}

\begin{defn}[{c.f.\ \cite[Definition 4.4.14]{andreasthomE}}] We call $\m$ the {\em asymptotic matrix homotopy category} of separable \cast-algebras.
\end{defn}

\begin{lem}[{Cuntz \cite[Proposition 3.1(a)]{MR899916}}]\label{lem Cuntz sum} For any separable \cast-algebras $B$ and $C$, the natural map 
	\begin{equation}
	B \star C \to B \times  C	
	\end{equation}
is an $\m$-equivalence.
\qed
\end{lem}

\begin{cor}\label{cor coprod} For any separable \cast-algebras $B, C$ and $D$, the natural map 
	\begin{equation}
	(B \otimes D) \star (C \otimes D) \to (B \star C) \otimes D	
	\end{equation}
is an $\m$-equivalence.
\end{cor}
\begin{proof} The following diagram is commutative
	\begin{equation}
	\xymatrix{
	(B \otimes D) \star (C \otimes D) \ar[r] \ar[d]& (B \star  C) \otimes D \ar[d]\\	
	(B \otimes D) \times (C \otimes D) \ar[r] & (B \times  C) \otimes D\\
	}.
	\end{equation}
The vertical maps are $\m$-equivalences by Lemma~\ref{lem Cuntz sum} and the bottom horizontal map is an isomorphism.
\end{proof}

\begin{notation} Let $B$ be a separable \cast-algebra. Following Cuntz, we write $qB$ for the kernel of the folding map $\xymatrix{B \star B \ar[r]^-{\id \star \id} & B}$. 
\end{notation}

We note that the short exact sequence
	\begin{equation}
	\xymatrix{
	0 \ar[r] & qB \ar[r] & B \star B \ar[r] & B \ar[r] & 0\\}
	\end{equation}
is {\em split-exact}.
\begin{prop}\label{prop qBD} For any separable \cast-algebras $B$ and $D$, the natural map
	\begin{equation}\label{map qBD}
	\sigma_{B, D}\colon q(B \otimes D) \to qB \otimes D
	\end{equation}
is an $\m$-equivalence.
\end{prop}
\begin{proof} Fix $A$ and let $F$ denote the functor $F(B) \coloneqq \m(A, B)$. %Then $F$ is a homotopy invariant, split exact, matrix stable functor with values in abelian groups.

We apply $F$ to the following commutative diagram of split-exact sequences:
	\begin{equation}
	\xymatrix{
	0 \ar[r] & q(B \otimes D) \ar[r] \ar[d]^{\sigma_{B, D}} & B \otimes D \star B \otimes D \ar[r] \ar[d]& B \otimes D \ar[r] \ar@{=}[d] & 0\\
	0 \ar[r] & qB \otimes D \ar[r] & (B \star B) \otimes D \ar[r] & B \otimes D \ar[r] & 0
	}.
	\end{equation}
By Corollary~\ref{cor coprod}, $F$ induces isomorphism on the middle map. Since $F$ is split exact, it follows that $F(\sigma_{B, D})$ is an isomorphism.
Now the proof follows from Yoneda's Lemma.
\end{proof}

\begin{rem} 
%\begin{enumerate}[(i)]
%\item 
Let $\Ho$ denote the {\em homotopy} category of \cast-algebras and let $\n$ denote the {\em matrix homotopy category} with morphisms
	\begin{equation}
	\n(A, B) \coloneqq \colim_{n}\Ho(A, B \otimes M_{n}).
	\end{equation}
Then, in Lemma~\ref{lem Cuntz sum} and Corollary~\ref{cor coprod}, we actually have $\n$-equivalences. However, the map $\sigma_{B, D}$ from Proposition~\ref{prop qBD} is not an $\n$-equivalence in general. For instance, let $T_{0}$ denote the {\em reduced} Toepliz algebra. Then $T_{0}$ is $K\!K$-contractible, hence $q(T_{0}) \otimes \K$ is contractible i.e.\ homotopy equivalent to the zero algebra $0$ (c.f.\ \cite{MR750677}). However, $q\C \otimes T_{0} \otimes \K$ has a non-trivial projection, hence not contractible. It follows that $\sigma_{\C, T_{0}}\colon q(T_{0}) \to q\C \otimes T_{0}$ is {\em not} an $\n$-equivalence.

Indeed, for any $A$ and $B$, we have a natural isomorphism 
	\begin{equation}
	\Ho(A, B \otimes \K) \cong \n(A, B \otimes \K).
	\end{equation}
%and	
%	\begin{equation}
%	\Ho(A, B \otimes \K) \cong \Ho(A \otimes \K, B \otimes \K).% \otimes \K) \cong \S(A \otimes \K, B \otimes \K).
%	\end{equation}
Hence if $f\colon A \to B$ is an $\n$-equivalence, then $f \otimes \id_{\K}\colon A \otimes \K \to B \otimes \K$ is a homotopy equivalence. 
\end{rem}

\begin{rem} Let $A$ and $B$ be separable \cast-algebras.
%\item %Let $\M$ denote the category with morphisms $\M(A, B) \coloneqq  \S(A, B \otimes \K)$. Then we have a natural functor $\m \to \M$. 
\begin{enumerate}[(i)]
\item  We have a natural isomorphism 
	\begin{equation}
	\S(A, B \otimes \K) \cong \m(A, B \otimes \K).
	\end{equation}
It follows that if $f \in \m(A, B)$ is an $\m$-equivalence, then $f \otimes \id_{\K}$ is an $\S$-equivalence from $A \otimes \K$ to $B \otimes \K$. 	
\item	Tensoring with $\K$ gives an isomorphism
	\begin{equation}
	\S(A, B \otimes \K) \cong \S(A \otimes \K, B \otimes \K).% \otimes \K) \cong \S(A \otimes \K, B \otimes \K).
	\end{equation}
\end{enumerate}

\end{rem}

\section{Matrix Homotopy Symmetry}

The following definition is inspired by \cite{MR1305073}.
\begin{defn}
A separable \cast-algebra $A$ is {\em matrix homotopy symmetric} if $\id_{A} \in \m(A, A)$ has an additive inverse: there is $n \ge 1$ and $\eta\colon A \to A \otimes M_{m}$ such that $i_{n,1} \oplus \eta\colon A \to A \otimes  M_{n+m}$ is null-homotopic.
\end{defn}

\begin{rem}\label{rem HS} 
\begin{enumerate}[(i)]
\item If the monoid $\m(A, A)$ is a group, then $A$ is matrix homotopy symmetric. Conversely, if $A$ is matrix homotopy symmetric, then $\m(A, B)$ and $\m(B, A)$ are abelian groups for any $B$.  
\item If $A$ is matrix homotopy symmetric, then so is $A \otimes D$ for any $D$.
\item If $A$ is $\m$-equivalent to $B$ and $A$ is matrix homotopy symmetric, then so is $B$.
\end{enumerate}
\end{rem}

\begin{ex}
\begin{enumerate}[(i)]
\item\label{ex SA} The algebra $\Sigma^{1}$ is matrix homotopy symmetric. In fact, the algebra $C_{0}(X)$, of continuous functions vanishing at the base point, is matrix homotopy symmetric for any based, {\em connected}, finite CW-complex $X$ (c.f.\ \cite[Proposition 3.1.3]{MR1066807} and the discussion preceding it).
\item\label{ex qA} The algebra $qB$ is matrix homotopy symmetric for any $B$, by taking $n =  m = 1$ and $\eta = \tau$ the flip-map on $qA$ (c.f.\ \cite[Proposition 1.4]{MR899916}).
\end{enumerate}
\end{ex}

\begin{notation} Let $B$ be a separable \cast-algebra.
Let $\pi_{B}\colon qB \to B$ denote the composition  
	\begin{equation}
	\pi_{B}\colon \xymatrix{
	qB\, \ar@{^{(}->}[r] & B \star B \ar[r]^-{\id \star 0} & B 	
	}.
	\end{equation}
\end{notation}

We remark that $q$ is functorial (for $\ast$-homomorphisms) and for any $\ast$-homomorphism $f\colon A \to B$, we have a commutative diagram
	\begin{equation}
	\xymatrix{
	qA \ar[r]^{q(f)} \ar[d]^{\pi_{A}} & qB \ar[d]^{\pi_{B}}\\
	A \ar[r]^{f} & B\\
	}.
	\end{equation}
	
From our point of view, the following is the key ingredient that underlies both Theorem~\ref{thm shulman} and Theorem~\ref{thm DL}.	
\begin{prop}\label{prop HS qA} Let $A$ be a separable \cast-algebra. Then the following statements are equivalent:
\begin{enumerate}[(a)]
\item\label{1} The algebra $A$ is matrix homotopy symmetric. 
\item\label{2} For any $B$ and $D$, we have  
	\begin{equation*}
	(\pi_{B} \otimes \id_{D})_{*}\colon \m(A, qB \otimes D) \cong \m(A, B \otimes D).	
	\end{equation*}
\item\label{3} The map $\pi_{A}\colon qA \to A$ is an $\m$-equivalence.
\item\label{4} The map $\pi_{\C} \otimes \id_{A}\colon q\C \otimes A  \to A$ is an $\m$-equivalence.
\end{enumerate}
\end{prop} 
\begin{proof}  The statements (\ref{3}) and (\ref{4}) are equivalent by Proposition~\ref{prop qBD}.

Since $qA$ is matrix homotopy symmetric (c.f.\ Example~\ref{ex qA}), it follows from Remark~\ref{rem HS} that (\ref{3}) implies (\ref{1}). 

Suppose that $A$ is matrix homotopy symmetric. Then the functor $F(B) \coloneqq \m(A, B \otimes D)$ is a homotopy invariant, split exact, matrix stable functor with values in abelian groups. Hence $(\pi_{B})_{*}\colon F(qB) \to F(B)$ is an isomorphism for all $B$, by \cite[Proposition 3.1]{MR899916}, i.e.\ (\ref{1}) implies (\ref{2}).

The remaining implication, (\ref{2}) $\Rightarrow$ (\ref{3}), % and (\ref{2}) $\Rightarrow$ (\ref{4}), 
follows from Yoneda's Lemma.
\end{proof}

As a corollary, we now prove Theorem~\ref{thm shulman}. In view of Proposition~\ref{prop qBD}, it is enough to prove the following. See also Theorem~\ref{thm con Shul}.
\begin{thm}[Bott Periodicity]\label{thm BP}
Let $u\colon q\C \to \Sigma^{2} \otimes M_{2} \in \m(q\C, \Sigma^{2})$ denote the Bott element. Then
 	\begin{equation}
	u \otimes \id_{\K}\colon q\C \otimes \K \to \Sigma^{2} \otimes M_{2} \otimes \K \cong \Sigma^{2} \otimes \K
	\end{equation}
is an $\m$-equivalence (equivalently, an $\S$-equivalence).	
\end{thm}
\begin{proof} We have a commutative diagram
	\begin{equation}
	\xymatrix{
	q(q\C) \ar[r]^{q(u)} \ar[d]^{\pi_{q\C}} & q(\Sigma^{2} \otimes M_{2}) \ar[d]^{\pi_{\Sigma^{2} \otimes M_{2}}}\\
	q\C \ar[r]^{u} & \Sigma^{2} \otimes M_{2}\\
	}.
	\end{equation}
The vertical maps are $\m$-equivalences\footnote{The map $\pi_{q\C}$ is in fact an $\n$-equivalence by \cite[Theorem 1.6]{MR899916}.} by Proposition~\ref{prop HS qA} and the map $q(u) \otimes \id_{\K}$ is a {\em homotopy} equivalence (in particular, an $\m$-equivalence) by $K\!K$-theoretic Bott Periodicity. It follows that $u \otimes \id_{\K}$ is an $\m$-equivalence.
\end{proof}

%For any $B$ and $D$, there is a natural map $q(B \otimes D) \to qB \otimes D$. 

%For $\phi\colon qA \to B$,  we write $\sigma_{D}(\phi) $ for the composition $q(A \otimes D) \to qA \otimes D \to B \otimes D$.

%\begin{cor} Let $A$ be a separable \cast-algebra. If $A$ is matrix homotopy symmetric, then for any $B \inn \S$,
%	\begin{equation}
%	\S(A, B \otimes \K) \cong \S(A, \Sigma^{2}B \otimes \K).
%	\end{equation}
%In particular, if $A$ is stable, then $A$ is matrix homotopy symmetric if and only if $A \cong_{\S} \Sigma^{2}A$.
%for any $A$ and $B$,
%	\begin{equation}
%	E(A, B) \cong \S(\Sigma^{2}A, B \otimes \K).
%	\end{equation}	
%\end{cor}
%\begin{proof}	
%Moreover, for any $B$, using Proposition~\ref{prop HS qA}(\ref{1}) $\Rightarrow$ (\ref{2}) with $D = \K$ and Proposition \ref{prop Shul}, we see that
%	\begin{equation}
%	\m(A, B \otimes \K) \cong \m(A, qB \otimes \K) \cong \m(A, \Sigma^{2}B \otimes \K).  
%	\end{equation}
%\end{proof}

\section{Bott Invertibility}

\begin{defn}
Let $u\colon q\C \to \Sigma^{2} \otimes M_{2} \in \m(q\C, \Sigma^{2})$ denote the Bott element. % and let $v\colon \Sigma^{2} \to \K$ denote the inverse Bott element.
We say that a separable \cast-algebra $D$ is {\em Bott inverting} if the element
	\begin{equation}
	u \otimes \id_{D} \in \m(q\C \otimes D, \Sigma^2 \otimes D)	
	\end{equation}
is an $\m$-equivalence.
\end{defn}

%\begin{thm}[Bott Periodicty]\label{thm Bott Periodicty} 
%Following diagram is commutative:
%	\begin{equation}
%	\xymatrix{
%	q\C \ar[r]^{u} \ar[d]^{\pi_{\C}} & \Sigma^{2} \otimes M_{2}\ar[d]^{M_{2}v}\\
%	\C \ar[r]^{i_{2} \otimes i_{\infty}} & M_{2} \otimes \K\\
%	}
%	\end{equation}
%$\Sigma^{1} v$ is an $\S$-equivalence.
%$u \otimes \id_{\K}\colon q\C \otimes \K \to \Sigma^{2} \otimes \K$ is an $\S$-equivalence.
%\end{thm}
%  
%Then $\sigma_{\Sigma^{1}}(u): q\Sigma^{1} \to M_{2}\Sigma^{3}$ factors as $q\Sigma^{1} \to \Sigma^{1} \to M_{2}\Sigma^{3}$.

\begin{rem}\label{rem BI} 
\begin{enumerate}[(i)]
\item If $D$ is Bott inverting, then so is $D \otimes B$ for any $B$.
\item If $D$ is $\m$-equivalent to $B$ and $D$ is Bott inverting, then so is $B$.
\end{enumerate}
\end{rem}

First we show that there are plenty of algebras that are Bott inverting. See Example~\ref{ex not BI} for an example that is {\em not} Bott inverting.%, following A.\ Thom \cite{andreasthomE}. 
\begin{lem}\label{lem Bott} Let $D$ be a separable \cast-algebra. Suppose that for some $n \ge 1$, the inclusion 
	\begin{equation}
	\id_{D} \otimes i_{n,1}\colon D \incl D \otimes M_{n}	
	\end{equation}
factors in $\S$ through a Bott inverting algebra . Then $D$ is Bott inverting.
\end{lem}
\begin{proof} Let 
	\begin{equation}
	\xymatrix{D \ar[r]^{f}& B \ar[r]^{g} & D \otimes M_{n}}
	\end{equation}
be a factorization of the inclusion $\id_{D} \otimes i_{n,1}\colon D \incl D \otimes M_{n}$, with $B$ Bott inverting. Then the following diagram is commutative in $\S$:
	\begin{equation}
	\xymatrix{
	q\C \otimes D \ar[r]^{\id_{q\C} \otimes f}\ar[d]^{u \otimes \id_{D}} & q\C \otimes B \ar[r]^{\id_{q\C} \otimes g} \ar[d]^{u \otimes \id_{B}} & q\C \otimes D \otimes M_{n} \ar[d]^{u \otimes \id_{D \otimes M_{n}}}\\
	\Sigma^{2} \otimes M_{2} \otimes D \ar[r]^{\id_{\Sigma^{2} \otimes M_{2}} \otimes f}& \Sigma^{2} \otimes M_{2} \otimes B \ar[r]^{\id_{\Sigma^{2} \otimes M_{2}} \otimes g} & \Sigma^{2} \otimes M_{2} \otimes D \otimes M_{n}}.
	\end{equation}
Since $i_{n,1}$ is invertible in $\m$, and $u \otimes \id_{B}$ is invertible by assumption, it follows that $u \otimes \id_{D}$ is invertible.
\end{proof}

\begin{defn}
We say that a \cast-algebra $D$ is {\em stable} if $D \cong D \otimes \K$.  
\end{defn}

By Bott Periodicity (Theorem~\ref{thm BP}) and Remark~\ref{rem BI}, stable algebras are Bott inverting.
\begin{lem}[{Kirchberg}]\label{lem Kirchberg} Let $D$ be a separable \cast-algebra. If $D$ contains a stable full \cast-subalgebra, then map
	\begin{equation}
	\id_{D} \otimes i_{4,1}\colon D \incl D \otimes M_{4}
	\end{equation}
factors through a stable algebra. 
\end{lem}
\begin{proof} See the proof of \cite[Lemma 4.4.7]{andreasthomE}.
\end{proof}

Combining Lemma~\ref{lem Bott} and Lemma~\ref{lem Kirchberg}, we get the following.
\begin{cor} All separable \cast-algebras that contain a stable full \cast-subalgebra are Bott inverting. In particular, all separable unital properly infinite \cast-algebras are Bott inverting. \qed
\end{cor}

\begin{rem} Same methods show that comparison map from algebraic to topological $K$-theory 
	\begin{equation}
	K^\mathrm{alg}_{*}(D) \to K^\mathrm{top}_{*}(D)
	\end{equation}
is an isomorphism if $D$ has a stable full \cast-subalgebra (c.f.\ \cite{MR1081504}).
\end{rem}

%
%Moreover %$\m(A, B \star D) \cong \m(A, B \times D)$ and 
%$\S(A, B \otimes \K) \cong \m(A, B \otimes \K) \cong \m(A \otimes \K, B \otimes \K)$.

Now we are ready to state and prove the connective versions of Theorem~\ref{thm shulman} and Theorem~\ref{thm DL}, which we recover by considering stable algebras.

\begin{thm}\label{thm con Shul}
Let $A$ be a separable \cast-algebra. If $D$ is Bott inverting, then we have equivalences
	\begin{equation}
	qA \otimes D \cong_{\m} q\C \otimes A \otimes D \cong_{\m} \Sigma^{2} \otimes A \otimes D.
	\end{equation}
\end{thm}
\begin{proof}
Follows from Proposition~\ref{prop qBD} and Bott invertibility.
\end{proof}

\begin{defn}[{A.\ Thom \cite[Theorem 4.2.1]{andreasthomE}}] Let $A$ and $B$ be separable \cast-algebras. For $n \in \Z$, we define $\bu_{n}(A, B)$ as the colimit
	\begin{equation}
	\bu_{n}(A, B) \coloneqq \colim_{k \to \infty}\m(\Sigma^{k} \otimes A, \Sigma^{k+n} \otimes B)	
	\end{equation}
along the suspension maps. The {\em connective $E$-category} $\bu$ is the category with morphisms $\bu_{0}(A, B)$.
\end{defn}

Let $X$ and $Y$ be based, connected, finite CW-complexes. Then from the proof of Theorem \cite[Theorem 4.2.1]{andreasthomE}, we see that 
	\begin{equation}
	\bu_{n}(C_{0}(X), C_{0}(Y)) \cong \kk_{n}(Y, X)
	\end{equation}
in the notation of \cite{MR1066807, MR1800209}.

\begin{ex}\label{ex not BI}
%\begin{enumerate}
%\item 
Let $X$ be a based, connected, finite CW-complex and let $D = C_{0}(X)$. Then, for any $k \le 0$, we have $\bu_{k}(D, \C) \cong 0$ by \cite[Corollary 3.4.3]{MR1066807}.

We claim that $D$ is Bott inverting if and only if $D$ is $\m$-contractible. Indeed, first note that, by Proposition~\ref{prop HS qA}, the map 
	\begin{equation}
	\id_{\Sigma^{1}} \otimes \pi_{\C} \colon \Sigma^{1} \otimes q\C  \to \Sigma^{1}	
	\end{equation}
is an $\m$-equivalence, thus $\pi_{\C}\colon q\C \to \C$ is a $\bu$-equivalence. Now suppose that $D$ is Bott inverting. Then   
	\begin{equation}
	\bu_{k}(D, \C) \cong \bu_{k}(q\C \otimes D, \C) \cong \bu_{k-2}(D, \C).	
	\end{equation}
for any $k \in \Z$. Hence the map $0\colon D \to 0$ induces an $\m$-equivalence by \cite[Theorem 2.4]{MR1800209}. The converse is clear. 

In particular, for any $k \ge 0$, the algebra $\Sigma^{k}$ is {\em not} Bott inverting.
%\end{enumerate}
\end{ex}

\begin{thm}\label{thm con DL} Let $A$ and $B$ be a Bott inverting separable \cast-algebras. If $A$ is matrix homotopy symmetric, then we have a natural isomorphism 
	\begin{equation}
	\m(A, B) \cong \bu(A, B).
	\end{equation}
\end{thm}
\begin{proof} Suppose that $A$ is matrix homotopy symmetric. By Proposition~\ref{prop HS qA}, we have isomorphisms
%	\begin{equation}
%	\xymatrix{
%	 \m(A, q\C \otimes B) \ar[d]_-{(\pi_{\C} \otimes \id_{B})_{*}}^{\cong} \ar[r]^{(\pi_{\C} \otimes \id_{A})^{*}}_{\cong} &\m(q\C \otimes A, q\C \otimes B) \ar[d]^{(\pi_{\C} \otimes \id_{B})_{*}}_{\cong}\\
%	\m(A, B)\ar[r]_{(\pi_{\C} \otimes \id_{A})^{*}}^{\cong} & \m(q\C \otimes A, B)\\}
%	\end{equation}
	\begin{equation}
	\xymatrix{
	 \m(A, q\C \otimes B) \ar[d]^-{\cong} \ar[r]^-{\cong} &\m(q\C \otimes A, q\C \otimes B) \ar[d]^-{\cong}\\
	\m(A, B)\ar[r]^-{\cong} & \m(q\C \otimes A, B)\\}.
	\end{equation}
and by Bott invertibility, we have 
	\begin{equation}
	\m(q\C \otimes A, q\C \otimes B) \cong \m(\Sigma^2 \otimes  A, \Sigma^2 \otimes  B). 	
	\end{equation}
Now it is easy to check that the composition 
	\begin{equation}
	\m(A, B) \to \m(q\C \otimes A, q\C \otimes B) \to \m(\Sigma^2 \otimes  A, \Sigma^2 \otimes B)
	\end{equation}
is the double suspension $\Sigma^2$.	
\end{proof}

\begin{cor} Let $A$ and $B$ be a matrix homotopy symmetric, Bott inverting separable \cast-algebras. Then $A$ and $B$ are $\bu$-equivalent if and only if they are $\m$-equivalent. 
\qed
\end{cor}

\bibliographystyle{amsalpha}
\bibliography{../BibTeX/biblio}

\def\romsup#1{{\edef\next{\the\font}$^{\next#1}$}} \def\cprime{$'$}
  \def\cftil#1{\ifmmode\setbox7\hbox{$\accent"5E#1$}\else
  \setbox7\hbox{\accent"5E#1}\penalty 10000\relax\fi\raise 1\ht7
  \hbox{\lower1.15ex\hbox to 1\wd7{\hss\accent"7E\hss}}\penalty 10000
  \hskip-1\wd7\penalty 10000\box7} \def\Dbar{\leavevmode\lower.6ex\hbox to
  0pt{\hskip-.23ex \accent"16\hss}D}
  \def\cfac#1{\ifmmode\setbox7\hbox{$\accent"5E#1$}\else
  \setbox7\hbox{\accent"5E#1}\penalty 10000\relax\fi\raise 1\ht7
  \hbox{\lower1.15ex\hbox to 1\wd7{\hss\accent"13\hss}}\penalty 10000
  \hskip-1\wd7\penalty
  10000\box7}\def\cfudot#1{\ifmmode\setbox7\hbox{$\accent"5E#1$}\else
  \setbox7\hbox{\accent"5E#1}\penalty 10000\relax\fi\raise 1\ht7
  \hbox{\raise.1ex\hbox to 1\wd7{\hss.\hss}}\penalty 10000 \hskip-1\wd7\penalty
  10000\box7} \def\polhk#1{\setbox0=\hbox{#1}{\ooalign{\hidewidth
  \lower1.5ex\hbox{`}\hidewidth\crcr\unhbox0}}}
  \def\cydot{\leavevmode\raise.4ex\hbox{.}}
  \def\cfgrv#1{\ifmmode\setbox7\hbox{$\accent"5E#1$}\else
  \setbox7\hbox{\accent"5E#1}\penalty 10000\relax\fi\raise 1\ht7
  \hbox{\lower1.05ex\hbox to 1\wd7{\hss\accent"12\hss}}\penalty 10000
  \hskip-1\wd7\penalty 10000\box7}
  \def\uarc#1{\ifmmode{\lineskiplimit=0pt\oalign{$#1$\crcr
  \hidewidth\setbox0=\hbox{\lower1ex\hbox{{\rm\char"15}}}\dp0=0pt
  \box0\hidewidth}}\else{\lineskiplimit=0pt\oalign{#1\crcr
  \hidewidth\setbox0=\hbox{\lower1ex\hbox{{\rm\char"15}}}\dp0=0pt
  \box0\hidewidth}}\relax\fi}
\providecommand{\bysame}{\leavevmode\hbox to3em{\hrulefill}\thinspace}
\providecommand{\MR}{\relax\ifhmode\unskip\space\fi MR }
\providecommand{\MRhref}[2]{%
  \href{http://www.ams.org/mathscinet-getitem?mr=#1}{#2}
}
\providecommand{\href}[2]{#2}
\begin{thebibliography}{GHT00}

\bibitem[CH90]{MR1065438}
Alain Connes and Nigel Higson, \emph{D\'eformations, morphismes asymptotiques
  et {$K$}-th\'eorie bivariante}, C. R. Acad. Sci. Paris S\'er. I Math.
  \textbf{311} (1990), no.~2, 101--106. \MR{MR1065438 (91m:46114)}

\bibitem[Cun84]{MR750677}
Joachim Cuntz, \emph{{$K$}-theory and {$C\sp{\ast} $}-algebras}, Algebraic
  $K$-theory, number theory, geometry and analysis (Bielefeld, 1982), Lecture
  Notes in Math., vol. 1046, Springer, Berlin, 1984, pp.~55--79. \MR{MR750677
  (86d:46071)}

\bibitem[Cun87]{MR899916}
\bysame, \emph{A new look at {$KK$}-theory}, $K$-Theory \textbf{1} (1987),
  no.~1, 31--51. \MR{MR899916 (89a:46142)}

\bibitem[DL94]{MR1305073}
Marius D{\u{a}}d{\u{a}}rlat and Terry~A. Loring, \emph{{$K$}-homology,
  asymptotic representations, and unsuspended {$E$}-theory}, J. Funct. Anal.
  \textbf{126} (1994), no.~2, 367--383. \MR{MR1305073 (96d:46092)}

\bibitem[DM00]{MR1800209}
Marius Dadarlat and James McClure, \emph{When are two commutative
  {$C^*$}-algebras stably homotopy equivalent?}, Math. Z. \textbf{235} (2000),
  no.~3, 499--523. \MR{MR1800209 (2001k:46080)}

\bibitem[DN90]{MR1066807}
M.~D{\u{a}}d{\u{a}}rlat and A.~N{\'e}methi, \emph{Shape theory and (connective)
  {$K$}-theory}, J. Operator Theory \textbf{23} (1990), no.~2, 207--291.
  \MR{MR1066807 (91j:46092)}

\bibitem[GHT00]{MR1711324}
Erik Guentner, Nigel Higson, and Jody Trout, \emph{Equivariant {$E$}-theory for
  {$C\sp *$}-algebras}, Mem. Amer. Math. Soc. \textbf{148} (2000), no.~703,
  viii+86. \MR{MR1711324 (2001c:46124)}

\bibitem[Shu10]{MR2606883}
Tatiana Shulman, \emph{The {$C^*$}-algebras {$qA\otimes \scr K$} and
  {$S^2A\otimes\scr K$} are asymptotically equivalent}, J. Operator Theory
  \textbf{63} (2010), no.~1, 85--100. \MR{2606883 (2011d:46146)}

\bibitem[SW90]{MR1081504}
Andrei~A. Suslin and Mariusz Wodzicki, \emph{Excision in algebraic {$K$}-theory
  and {K}aroubi's conjecture}, Proc. Nat. Acad. Sci. U.S.A. \textbf{87} (1990),
  no.~24, 9582--9584. \MR{MR1081504 (91j:19008)}

\bibitem[Tho03]{andreasthomE}
Andreas Thom, \emph{Connective {$E$}-theory and bivariant homology for
  {$C^*$}-algebras}, Ph.D. thesis, Muenster, 2003.

\bibitem[Wei94]{MR1269324}
Charles~A. Weibel, \emph{An introduction to homological algebra}, Cambridge
  Studies in Advanced Mathematics, vol.~38, Cambridge University Press,
  Cambridge, 1994. \MR{MR1269324 (95f:18001)}

\end{thebibliography}

\end{document}